\documentclass[11pt,reqno]{amsart}
\usepackage{fullpage,amssymb,amsmath,verbatim,amstext,eucal}
\usepackage[all]{xy}
\numberwithin{equation}{section}

\newcommand{\numberby}{equation}
\numberwithin{equation}{section}

\newtheorem{lemma}[\numberby]{Lemma}

\newtheorem{proposition}[\numberby]{Proposition}
\newtheorem{theorem}[\numberby]{Theorem}

\newtheorem{corollary}[\numberby]{Corollary}

\newcommand{\FLEX}{\relax}
\newcommand{\flex}[1]{\renewcommand{\FLEX}{#1}}
\newtheorem{flexthm}[\numberby]{\FLEX}
\newenvironment{flexstate}[2]{\flex{#1}\begin{flexthm}[#2]}{\end{flexthm}}

\long\def\quiet #1 \endquiet{\relax}
\long\def\Quiet #1 \endQuiet{\relax}

\usepackage{xcolor}

\theoremstyle{definition}
\newtheorem{definition}[\numberby]{Definition}
\newtheorem{example}[\numberby]{Example}

\newenvironment{remark}[1]{\refstepcounter{\numberby}%
\vskip 5pt \par\noindent {\bf #1\ \thelemma .}}{\vskip 5pt \par}

\newenvironment{remark*}[1]{\par \vskip 5pt \noindent 
{\bf #1.}}{\vskip 5pt \par}

\newlength{\dqlength}

\newcommand{\bh}{\ensuremath{{\mathcal B}({\mathcal H})}}

\newcommand{\cstaralg}{$C^*$-algebra}

\providecommand{\dual}[1]{\ensuremath{#1^{\#}}}

\newcommand{\dperp}{{\perp\perp}}

\newcommand{\dstext}[1]{\quad\text{#1}\quad}

\newcommand{\id}{{\operatorname{id}}}

\newcommand{\idealin}{\unlhd}

\providecommand{\qed}%
{\hfill \vrule height5pt width4pt depth1pt \vspace{+2.00ex}}

\newcommand{\supp}{\operatorname{supp}}
\newcommand{\tnorm}[1]{{\left\vert\kern-0.25ex\left\vert\kern-0.25ex\left\vert #1 
   \right\vert\kern-0.25ex\right\vert\kern-0.25ex\right\vert}}

  \newcommand{\A}{{\mathcal{A}}}
  
  \newcommand{\B}{{\mathcal{B}}}
  \newcommand{\C}{{\mathcal{C}}}
  
  \newcommand{\D}{{\mathcal{D}}}

  \newcommand{\G}{{\mathcal{G}}}

\let\aS=\S  
\renewcommand{\S}{{\mathcal{S}}}

  \newcommand{\V}{{\mathcal{V}}}

\newcommand{\ideal}{\operatorname{ideal}}
\providecommand{\dperp}{{\perp\perp}}
\newcommand{\intr}{\operatorname{int}}
\newcommand{\cl}{\operatorname{cl}}
\newcommand{\ropen}{\operatorname{\textsc{Ropen}}}
\newcommand{\rideal}{\operatorname{\textsc{Rideal}}}

\newcommand{\isol}{\text{\sc isol}}

\begin{document}

\title{Irreducible Maps and Isomorphisms of Boolean Algebras of
               Regular Open Sets and Regular Ideals}

\author[D.R. Pitts]{David
R. Pitts} \address{Dept. of Mathematics\\ University of
Nebraska-Lincoln\\ Lincoln, NE\\ 68588-0130}
\email{dpitts2@unl.edu}

\begin{abstract}
 Let $\pi: Y\rightarrow X$ be a continuous surjection between compact
 Hausdorff spaces $Y$ and $X$ which is irreducible in the sense that if
 $F\subsetneq Y$ is closed, then $\pi(F)\neq X$.   We exhibit  isomorphisms between
  various  Boolean algebras associated to this data:  the regular open
  sets of $X$, the regular 
  open sets of  $Y$, the regular ideals of $C(X)$ and the regular
  ideals of $C(Y)$.

  We call $X$ and $Y$ Boolean equivalent if the
  regular open sets of $X$ and the regular open sets of $Y$ are isomorphic
  Boolean algebras.  We give a characterization of when two compact
  metrizable spaces are Boolean equivalent; this
  characterization may be viewed as a topological version of the characterization of
  standard Borel spaces.
  \end{abstract}

\maketitle

\section{Introduction}
A continuous surjection $\pi: Y\rightarrow X$ between compact
Hausdorff spaces is called \textit{irreducible} if the only closed
subset of $Y$ which surjects onto $X$ is $Y$ itself.
The map $\pi$ dualizes to a $*$-monomorphism $\alpha: C(X)\rightarrow
C(Y)$, given by $f\mapsto f\circ\pi$.  As $\pi$ is irreducible,
$\alpha$ has the property that for every non-zero ideal $K\idealin
C(Y)$, $\alpha^{-1}(K)\idealin C(X)$ is non-zero (see Lemma~\ref{ecoveex}
 for a proof).
Associated to this data are: 
\begin{enumerate}
\item \label{BAXY} the Boolean algebras,  $\ropen(X)$ and $\ropen(Y)$, of regular open
  sets of $X$ and $Y$; and 
  \item \label{dBAXY} the Boolean algebras, $\rideal(C(X))$ and $\rideal(C(Y))$, of
    regular ideals of $C(X)$ and $C(Y)$.
  \end{enumerate}
  It is well-known that closed ideals in $C(X)$ are in bijective
  correspondence with open subsets of $X$. Restricting this bijection to 
  $\rideal(X)$ gives a Boolean algebra isomorphism between
  $\rideal(C(X))$ and $\ropen(X)$; this bit of folklore is recorded in
  Lemma~\ref{ridisorop} below.

  A far less familiar fact is that all four of the Boolean algebras
  listed above are isomorphic.  In~\cite[Lemma~2.13]{PittsStReInII},
  we stated  without proof that the pairs of Boolean algebras given in
 \eqref{BAXY} and \eqref{dBAXY} are isomorphic.  However,
  \cite[Lemma~2.13]{PittsStReInII} incorrectly describes the
  isomorphism of $\ropen(X)$ onto $\ropen(Y)$.  Aside from
  ~\cite{PittsStReInII}, we are not aware of other publications where the
  isomorphism of these Boolean algebras is asserted.
   
  The main purpose of these notes is to provide
  complete proofs that the four Boolean algebras listed above are
  isomorphic and to explicitly describe isomorphisms (in terms of
  $\pi$ and $\alpha$) between them.  Propositions~\ref{YXiso}
  and~\ref{ridealISO} below give  isomorphisms between the
  Boolean algebras listed above in~\eqref{BAXY} and ~\eqref{dBAXY}
  respectively.  As Lemma~\ref{ridisorop} gives an 
  isomorphism (and its inverse) between $\ropen(X)$ and
  $\rideal(C(X))$, all four of these Boolean algebras are isomorphic.
  While Propositions~\ref{YXiso}
  and~\ref{ridealISO} give more detailed information, we summarize
  them here. 
    \begin{theorem}\label{sumres}   Let $\pi: Y\rightarrow X$ be an
      irreducible map, and let $\alpha:
      C(X)\rightarrow C(Y)$ be the $*$-monomorphism dual to $\pi$. The following statements hold.
      \begin{enumerate}
\item \label{sumres1} The map $\ropen(Y) \ni U
    \mapsto \intr(\pi(\cl(U)))\in \ropen(X)$ is 
  is an isomorphism 
  of Boolean algebras whose inverse is 
  $\ropen(X)\ni V\mapsto \intr(\cl(\pi^{-1}(V)))\in\ropen(Y)$.
\item \label{sumres2}
  The map
  $\rideal(C(X))\ni J\mapsto \alpha(J)^\dperp\in \rideal(C(Y))$
is a Boolean algebra isomorphism whose inverse
is 
 $\rideal(C(Y))\ni
  K\mapsto \alpha^{-1}(K)\in \rideal(C(X))$.
      \end{enumerate}
    \end{theorem}
  
Our route to Theorem~\ref{sumres} starts in Section~\ref{sec:pc}, where
we establish part~\eqref{sumres1} in the
special case when the domain of $\pi$  is a projective space;
Section~\ref{sec:ECI} does the  general case.  In
Section~\ref{sec:BARI}, we explain the notation in
Theorem~\ref{sumres}\eqref{sumres2} and give its proof. 
  
Section~\ref{SecBE} gives an application of Theorem~\ref{sumres}: for
compact and metrizable spaces $X$ and $Y$, Theorem~\ref{compactmetric}
characterizes when $\ropen(X)$
and $\ropen(Y)$ are isomorphic Boolean algebras.  In particular, when $X$ and $Y$ are perfect, compact and
metrizable, their Boolean algebras of regular open sets are
isomorphic.

  Our interest in Theorem~\ref{sumres} originated with our studies
  in~\cite{PittsStReInI,PittsStReInII} of regular inclusions of
  \cstaralg s.  When $\D$ is a unital  \cstaralg\ embedded
  as a  subalgebra of the unital \cstaralg\ $\C$, there is
  always a unital completely positive mapping $E$ of $\C$ into the injective
  envelope $I(\D)$ of $\D$ which extends the inclusion of $\D$ into
  $I(\D)$.  We called such a map a pseudo-expectation for the
  inclusion.  In some cases, a pseudo-expectation is unique or both 
  faithful and unique.  Such properties of pseudo-expectations imply
  interesting structural properties of the inclusion $\D\subseteq\C$,
  some of which may be found in~\cite{PittsStReInI,PittsZarikianUnPsExC*In}.
  In the  abelian case,~\cite[Corollary~3.22]{PittsZarikianUnPsExC*In}
  characterizes the inclusions $C(X)\subseteq C(Y)$ having a unique
  and faithful pseudo-expectation as those for which the associated
  surjection of $Y$ onto $X$ is irreducible.  This fact
  played an instrumental role in our characterization in~\cite{PittsStReInII} of regular
  inclusions having a Cartan envelope.

  We thank Jon Brown, Ruy Exel, Adam Fuller, and Sarah Reznikoff for
  several helpful conversations.

\section{Projective Covers of Compact Hausdorff Spaces}\label{sec:pc}

Throughout, all topological spaces are assumed Hausdorff
and compact, and all maps between spaces are assumed continuous.  When
$X$ is a space and $E\subseteq X$ is any subset, we will use $\cl E$
and $\overline E$ interchangeably for the closure of $E$; likewise we
will use $\intr E$ and $E^\circ$ interchangeably for the interior of
$E$.

The main goal of this section is to establish
Proposition~\ref{fineform}, which describes an isomorphism between the
Boolean algebras of regular open sets for spaces $X$ and $P$ when $f:
P\rightarrow X$ is an irreducible map and $P$ is a projective space. 

We begin by recalling some facts about projective
topological spaces and projective covers of compact Hausdorff spaces. 
Gleason~\cite{GleasonPrToSp} calls the space $P$ \textit{projective}
if given spaces $Y$ and $X$, a surjective map
$\pi: Y\twoheadrightarrow X$ and a map $f: P\rightarrow X$, there
exists a map $g: P\rightarrow X$ such that $f=\pi\circ g$.
\begin{equation}\label{CD}
  \xymatrix{Y\ar@{>>}[d]_\pi & \\ X & P\, .\ar@{>}[l]^f\ar@{-->}[ul]_g}
\end{equation}
The space $P$ is projective if and only it is extremally
disconnected (that is, the closure of every open set is open)~\cite[Theorem~2.5]{GleasonPrToSp}.   (A
projective, compact Hausdorff space is sometimes called a
\textit{Stonean space}.)

Following~\cite{HadwinPaulsenInPrAnTo},
a \textit{cover} for the  compact Hausdorff
space $X$ is a pair $(Y,\pi)$ consisting of a compact Hausdorff space $Y$
and a continuous surjection $\pi:Y\rightarrow X$.
If $\pi$ is irreducible, the cover $(Y,\pi)$ is called an
\textit{essential cover}. If the only continuous
map $h:Y\rightarrow Y$ which satisfies $\pi\circ h=\pi$ is
$h=\text{id}_Y$. the cover $(Y,\pi)$ of $X$  is called \textit{rigid}.
\begin{remark}{Remark} The definition of essential cover given here differs
from, but is easily seen to be
equivalent to, the definition given
in~\cite{HadwinPaulsenInPrAnTo}.
\end{remark}

By~\cite[Proposition~2.13]{HadwinPaulsenInPrAnTo}, if $(P,f)$ is a
cover for $X$ with $P$ a projective space, then $(P,f)$ is rigid if
and only if $(P,f)$ is essential.
\begin{definition}[{\cite{HadwinPaulsenInPrAnTo}}]
If $(P,f)$ is a rigid cover for $X$ and $P$ is projective, $(P,f)$  is called a
\textit{projective cover} for $X$.   Projective covers are also called
\textit{Gleason
  covers}, see~\cite{BezhanishviliGabelaiaHardingJibladzeCoHaSpReGlSp}.
\end{definition}

For an open subset $V$ of $X$, let
\[V^\perp:=X\setminus \overline V=(X\setminus V)^\circ \dstext{and write}
V^\dperp:=(V^\perp)^\perp.    \] 
Recall that an open set $V\subseteq X$ is called a \textit{regular
  open set} if $V=V^\dperp$, that is, if $V=\intr(\cl V)$.
\begin{remark}{Notation}
We shall use $\ropen(X)$ to denote the family of all
regular open sets in $X$.
\end{remark}
Here are some facts regarding $\ropen(X)$.
\begin{flexstate}{Facts}{}\label{ropfacts}  Let $X$ be a compact
               Hausdorff space. 
  \begin{enumerate}
    \item Suppose $U$ is an open subset of $X$ and $x_0\in U$.  Then
               there is $V\in \ropen(X)$ with \[x_0\in V\subseteq
               \overline V\subseteq U.\]  In particular, $\ropen(X)$
               is a base for the topology of $X$. 

\item With the
operations ~
\begin{equation}\label{ropenops}  V_1\vee V_2:= \text{int}(\text{cl}(V_1\cup V_2)),\quad V_1\wedge
  V_2:= V_1\cap V_2, \dstext{and} \neg V:= V^\perp,
 \end{equation}
$\ropen(X)$ is a complete
Boolean algebra.
\end{enumerate}
\end{flexstate}
\begin{proof}
(a) 
As $X$ is a regular topological space,  there exist disjoint open subsets $V_1$ and $V_2$ with
$x_0\in V_1$ and $X\setminus U\subseteq V_2$. Then
$x_0\in V_1\subseteq \overline{V_1}\subseteq U$.  Take $V:=\intr(\cl
V_1))\in \ropen(X)$.

(b) See \cite[\aS 7, Lemma~1]{HalmosLeBoAl}.
\end{proof}

\begin{remark}{Standing Assumption and Notation} \label{stas1} With the exception of the
  material following Proposition~\ref{PhiDes},
  for the remainder of the section, we shall
fix the space $X$ and let $P$ be the dual of $\ropen(X)$, that is, $P$ is the set of all
Boolean algebra homomorphisms of $\ropen(X)$ into the Boolean algebra
$\mathbf{2}:= \{0,1\}$.

As $P\subseteq \mathbf 2^{\ropen(X)}$ is
closed (\cite[\aS 18, Lemma~2]{HalmosLeBoAl}),  $P$ is compact.
Furthermore~\cite[\aS 21, Theorem~10]{HalmosLeBoAl}, $P$ is extremally
disconnected (i.e.\ Stonean) and hence $P$ is projective.
\end{remark}

Note that $\ropen(P)$ is precisely
the collection of clopen subsets of $P$.  Furthermore, for $V\in
\ropen(X)$, define
\begin{equation}\label{phidef}
  \Phi(V):=\{p\in P: p(V)=1\}.
\end{equation}
Then $\Phi(V)$ is a
clopen subset of $P$, so $\Phi$ gives a mapping from $\ropen(X)$ into
$\ropen(P)$.  The following  is an application of 
 Stone's duality theorem.
\begin{flexstate}{Fact}{c.f.\ {\cite[\aS 18, Theorem 6]{HalmosLeBoAl}}}\label{StoneTh}  The map $\Phi:
  \ropen(X)\rightarrow\ropen(P)$ is an isomorphism of Boolean algebras.
\end{flexstate}

Gleason showed that any compact Hausdorff space $X$ has a (nearly
unique) projective cover; see Fact~\ref{pcu} below.
In~\cite{RainwaterNoPrRe}, Rainwater
gave a rather different
and simpler construction of a projective cover for $X$ (see also~\cite{HadwinPaulsenInPrAnTo}).  We now describe a slight
modification of Gleason's construction of a projective cover for $X$.
We refer the reader to \cite[Section~3]{GleasonPrToSp}, or
~\cite{StraussExDiSp} for the proof of the following result.

\begin{theorem}[Gleason]\label{pcover} Given $p\in P$, the collection, \[\V_p:=\{\overline V: V\in
\ropen(X), p(V)=1\}\] has the finite intersection property and
$\bigcap \V_p$ is a singleton set.   Let $f(p)$ be the
element of $\bigcap \V_p$.  
The function $f: P\rightarrow X$ is surjective, continuous and the pair $(P,f)$ is a
projective cover for $X$.
\end{theorem}  

Our next goal is to describe the inverse of the map $\Phi$ given
in~\eqref{phidef} in terms of the map $f$ given in
Theorem~\ref{pcover}; this is accomplished in Proposition~\ref{PhiDes}
below.
\begin{lemma}\label{onto}  Let $B\in\ropen(X)$.
  Then $B\subseteq
  f(\Phi(B))\subseteq \overline{B}$.
\end{lemma}
\begin{proof}
Let $p\in \Phi(B)$. Then $p(B)=1$, so $B\in\V_p$.  Thus 
$f(p)\in \overline B$, which gives
$f(\Phi(B))\subseteq \overline B$.   

We now show
$B\subseteq f(\Phi(B))$.  To do this, choose $x\in B$ and set 
\[\G_x:=\{G\in \ropen(X): x\in G\}.\]

By Fact~\ref{ropfacts}(a),  if $H\subseteq X$
is an open neighborhood of $x$, then there exists $H_1\in\ropen(X)$ such
that $x\in H_1\subseteq \overline H_1 \subseteq H$.  It follows
that  \begin{equation}\label{onto1} \bigcap \{\overline G:
  G\in \G_x\}=\{x\}.
\end{equation}

If $G_1, \dots, G_n\in\G_x$, then $\bigcap_{j=1}^n G_j$ is not the
zero element of $\ropen(X)$ because 
$x\in\bigcap_{j=1}^n G_j$.   Therefore,   
\[\Phi\left(\bigcap_{j=1}^n G_j\right)\stackrel{\eqref{StoneTh}}{=} \bigcap_{n=1}^n \Phi(G_j)\neq \emptyset.\]  Thus
$\{\Phi(G): G\in \G_x\}$ is a collection of clopen sets in $P$ having
the finite intersection property.  Fix
\[p\in \bigcap \{\Phi(G): G\in\G_x\}.\]  If $G\in
\G_x$,~\eqref{phidef} shows $p(G)=1$.   Thus
 $\G_x\subseteq \{U\in \ropen(X): p(U)=1\}$, so by definition of
 $f$, 
$f(p)\in \overline G$ for every $G\in
\G_x$. By~\eqref{onto1},  $f(p)=x$.   Since $B\in \G_x$, 
$p\in \Phi(B)$.  Thus
$x\in f(\Phi(B))$.
\end{proof}

We now describe the inverse of $\Phi$.   For $E\in \ropen(P)$,
let
\begin{equation}\label{psidef}
  \Psi(E):=\intr(f(E)).
\end{equation}
Since elements of $\ropen(P)$ are compact subsets of $P$ and $f$ is continuous,
$\Psi$ is a mapping of $\ropen(P)$ into   $\ropen(X)$. 
\begin{corollary}\label{ontocor}  With $\Phi$ and $\Psi$ defined as in
  \eqref{phidef} and \eqref{psidef}, $\Psi=\Phi^{-1}$.
\end{corollary}
\begin{proof}
  Let $B\in\ropen(X)$.   We claim $B=\Psi(\Phi(B))$.
  Since $f(\Phi(B))\subseteq \overline B$,
  \[\Psi(\Phi(B))\stackrel{\eqref{psidef}}{=}\intr(f((\Phi(B))))
    \stackrel{\eqref{onto}}{\subseteq}
    \intr(\overline B)=B.\] For the reverse inclusion,
  Lemma~\ref{onto} gives
  $B\subseteq f(\Phi(B))$, so
  \[B=\intr(B)\subseteq \intr(f(\Phi(B))=\Psi(\Phi(B)),\] establishing
  the claim.

  Thus $\Psi\circ \Phi=\id|_{\ropen(X)}$.  By Fact~\ref{StoneTh}, $\Phi$ is
an isomorphism, and therefore $\Psi=\Phi^{-1}$.
\end{proof}

The definition of $\Psi$ is in terms of $f$, but the definition of
$\Phi$ is not.  Here is a  description of $\Phi$ in terms of $f$.  

\begin{proposition}\label{PhiDes}   For each $V\in \ropen(X)$,
  \begin{equation} \label{PhiDes00}\Phi(V)=\cl(f^{-1}(V)).
  \end{equation}
\end{proposition}
\begin{proof}
  The first step is to show that for $V\in \ropen(X)$,
  \begin{equation}\label{PhiDes1}
    \cl(f^{-1}(V))\subseteq \Phi(V).
  \end{equation}
  To do this, it is convenient to verify that for all $V\in \ropen(X)$,
  \begin{equation}\label{PhiDes2}
    f^{-1}(\neg V) \subseteq \Phi(\neg V).
  \end{equation}
  Let $p\in f^{-1}(\neg V)$.  Then $f(p)\in \neg V=X\setminus \cl(V)$.
  Note that $p(\neg V)=1$: otherwise
$p(V)=1$, which by definition of $f$, leads to the conclusion
that $f(p)\in \cl(V)$, a contradiction.  
But \[\Phi(\neg V)=\{p\in P:
p(\neg V)=1\},\]  so $p\in \Phi(\neg V)$.   This gives~\eqref{PhiDes2}.

Replacing $\neg V$ with $V$ in~\eqref{PhiDes2} gives
\[f^{-1}(V)\subseteq \Phi(V).\]
Since $\Phi(V)$ is clopen,~\eqref{PhiDes1} follows.

Applying $\Psi$ to each side
of~\eqref{PhiDes1} yields
\begin{equation}\label{PhiDes3} \intr(f(\cl(f^{-1}(V))))\subseteq V.
\end{equation}  But
\[V=f(f^{-1}(V))\subseteq f(\cl(f^{-1}(V))),\]
so \begin{equation*}
  V\subseteq \intr(
  f(\cl(f^{-1}(V)))) \stackrel{\eqref{PhiDes3}}{\subseteq} V,
  \dstext{that is,} V = \intr(f(\cl(f^{-1}(V)))). 
\end{equation*}
Thus, 
\[\Psi(\cl (f^{-1}(V)))= V\stackrel{\eqref{ontocor}}{=}
  \Psi(\Phi(V)).\] Since $\Psi$ is
one-to-one, we obtain~\eqref{PhiDes00}.
\end{proof}

For the remainder of this section, we relax Standing
Assumption~\ref{stas1}:  $X$ will remain a compact Hausdorff space,
but we no longer assume that the projective space $P$ is the dual of
$\ropen(X)$.

While not unique, Gleason observed the projective cover is nearly
unique in  a sense we now explain. 
Suppose for $i=1,2$,  $(P_i,f_i)$ is a projective cover for for $X$.
For $i\neq j$, projectivity yields maps $\phi_{ij}: P_j\rightarrow P_i$ such that
\[f_j=f_i\circ \phi_{ij}.\]  Then $f_i\circ\phi_{ij}\circ
  \phi_{ji}=f_j\circ\phi_{ji}=f_i$, so rigidity of the cover $(P_i,
  f_i)$  yields
  \[\id_{P_1}=\phi_{12}\circ\phi_{21}\dstext{and}
    \id_{P_2}=\phi_{21}\circ\phi_{12},\] that is,
  $\phi_{12}=\phi_{21}^{-1}$.   Note that if $\phi_{21}$ is chosen, then any choice for
  $\phi_{12}$ is necessarily the inverse of $\phi_{21}$, so actually
  the $\phi_{ij}$ are unique.  These considerations give the following.
  \begin{flexstate}{Fact}{Gleason} \label{pcu}  Given projective covers $(P_i, f_i)$ for
    $X$, there exists a unique homeomorphism $\phi: P_1\rightarrow
    P_2$ such that $f_2\circ\phi=f_1$.
  \end{flexstate}

Fact~\ref{pcu} removes the need to use the particular projective
cover $(P,f)$ described in Theorem~\ref{pcover} when describing the maps $\Psi$
and $\Phi$.  Thus we obtain the following, which is the main result of this section. 
\begin{proposition}\label{fineform}   Let $(P,f)$ be any projective cover
  for the compact Hausdorff space $X$.  Define maps $\Phi:
  \ropen(X)\rightarrow\ropen(P)$ and $\Psi: \ropen(P)\rightarrow \ropen(X)$ by
  \begin{align}
    \Phi(V)&=\cl(f^{-1}(V)),& V\in \ropen(X) \label{fineform1}\\
    \intertext{and}
    \Psi(E)&=\intr(f(E)),&E\in\ropen(P). \label{fineform2}
  \end{align}
Then $\Phi$ and $\Psi$ are Boolean algebra isomorphisms and
$\Phi=\Psi^{-1}$.
\end{proposition}

We conclude this section with some  corollaries to Proposition~\ref{fineform}.
Recall that a closed set $F\subseteq X$ is a \textit{regular closed
  set} if $F=\cl(\intr(F))$.    Also, recall that if $B$ is a Boolean
algebra and $e\in B$, the \textit{relativization} of $B$ to $e$ is the
Boolean algebra, $B(e):=\{e\wedge x: x\in B\}$.

\begin{corollary}\label{reclBA}   Let $X$ be a compact Hausdorff
  space with projective cover $(P,f)$,  let  $F\subseteq X$ be a regular
  closed set, and let $V:=\intr(F)$.  Then $(\Phi(V), f|_{\Phi(V)})$ is
  a projective cover for $F$.

  Furthermore, if $\Phi_F:\ropen(F)\rightarrow
  \ropen(\Phi(V))$ is the isomorphism obtained by applying Proposition~\ref{fineform} to
  $F$ and $(\Phi(V), f|_{\Phi(V)})$, then $\Psi\circ \Phi_F$ is an isomorphism of
  $\ropen(F)$ onto the
  relativization of $\ropen(X)$ to $V$.
\end{corollary}
\begin{proof}
We start by establishing $f(\Phi(V))=F$.  Since $\Phi(V)$ is clopen and 
$f(\Phi(V))=f(\cl(f^{-1}(V)))\supseteq f(f^{-1}(V))=V$,
$f(\Phi(V))\supseteq \cl(V) =F$.  For the reverse
inclusion, let $x\in f(\Phi(V))$ and choose $y\in \Phi(V)$ with
$f(y)=x$.   Then there exists a net $y_\lambda\in f^{-1}(V)$ with
$y=\lim y_\lambda$.   So $x=\lim f(y_\lambda)\in \cl(V)$.

As every clopen subset of a projective space is
projective, $(\Phi(V), f|_{\Phi(V)})$ is a projective cover for
$F$.
Since 
$\ropen(\Phi(V))$ is the relativization of 
$\ropen(P)$ to $\Phi(V)$, $\Psi|_{\ropen(\Phi(V))}$ is an isomorphism
onto the relativization of $\ropen(X)$ to $V$.   So $\Psi\circ \Phi_F$
is an isomorphism of $\ropen(F)$ onto the relativization of
$\ropen(X)$ to $V$.
\end{proof}

Let $B$ be a complete Boolean algebra.  Let $a:=\bigvee\{x\in B:
x\text{ is an atom of $B$}\}$ and let $B(a)= \{a\wedge b: b\in B\}$ and
$B(\neg a)=\{b\wedge \neg a: b\in B\}$ be the relativizations of $B$ to
$a$ and $\neg a$ respectively.     For any $b\in B$, $b=(a\wedge
b)\vee(\neg a\wedge b)$, so  $B$ is isomorphic to the
direct product $B(a)\times B(\neg a)$ of an atomic Boolean algebra with an
atomless Boolean algebra.

We next show how  Proposition~\ref{fineform}
  provides insights into this  decomposition of    
 $\ropen(X)$.  
For any space $Y$, let $\isol(Y)$ be the set of
isolated points of $Y$ and put
\begin{equation}\label{atna}
  Y_a:=\cl(\isol(Y))\dstext{and} Y_c:= \cl((Y\setminus Y_a)).
\end{equation}
Then $Y_a$ and $Y_c$ are regular closed sets.  Notice that $\isol(Y)$
and $(Y\setminus Y_a)$ are disjoint open sets whose union is dense in
$X$.  Since $\intr(\cl(\isol(Y)))=\intr(Y_a)$ and
$\intr(\cl(Y\setminus Y_a))=\intr(Y_c)$, the sets $\intr(Y_a)$ and
$\intr(Y_c)$ are disjoint and have dense union.  Thus in the Boolean
algebra $\ropen(Y)$,
\begin{equation}\label{anegc}
   \intr(Y_c)
=\neg\intr(Y_a).
\end{equation}

In the following we use  notation found in Corollary~\ref{reclBA}.
 \begin{corollary}\label{Cor:D}
   Let $(P,f)$ be a projective cover for the compact Hausdorff space
   $X$.  The following
   statements hold.
   \begin{enumerate}
   \item\label{Cor:D3} $f|_{\isol(P)}$ is a bijection of $\isol(P)$
     onto $\isol(X)$. 
  \item\label{Cor:D2}  Define $\Theta:
       \ropen(X_a)\times\ropen(X_c)\rightarrow \ropen(X)$ by
       \[\Theta(W_a,W_c):=\Psi(\Phi_{X_a}(W_a))\vee\Psi(\Phi_{X_c}(W_c)),\quad
         (W_a,W_c)\in \ropen(X_a)\times \ropen(X_c).\] Then $\Theta$
       is an isomorphism of Boolean algebras.
   \item\label{Cor:D1} $\ropen(X_a)$ is an atomic Boolean algebra and
     $\ropen(X_c)$ is an atomless Boolean algebra.
   \item \label{Cor:D4}  $X_c$ is a perfect set.
   \end{enumerate}
 \end{corollary}
 \begin{proof}
\eqref{Cor:D3} For $x\in
  X$,
Proposition~\ref{fineform} implies
   $\{x\}\in \ropen(X)$ if and
   only if $\{f^{-1}(x)\}\in \ropen(P)$.  As a singleton subset of a
   compact Hausdorff space is clopen
   if and only if the element it contains is an isolated point, 
   $f|_{\isol(P)}$ is a bijection of $\isol(P)$ onto  $\isol(X)$. 

\eqref{Cor:D2}  Let $V_a:=\intr(X_a)$ and $V_c:=\intr(X_c)$.
By~\eqref{anegc}, $V_c=\neg V_a$.  Now apply
Corollary~\ref{reclBA}.

\eqref{Cor:D1}  By construction, $\isol(X)$  is dense
in $X_a$.  For each $x\in \isol(X)$, $\{x\}$ is an atom of
$\ropen(X_a)$.  So $\bigvee\{\{x\}: x\in \isol(X)\}$ is the unit of
$\ropen(X_a)$.  Therefore, $\ropen(X_a)$ is atomic.

To show $\ropen(X_c)$ is atomless, we argue by contradiction.  If $e$ is an atom
of $\ropen(X_c)$, then there exists $x\in X_c$ so that
$e=\{x\}\in\ropen(X_c)$.  Then $\Theta(0,e)=\{x\}$ is an atom of $\ropen(X)$, so
 $x\in \isol(X)\cap X_c=\emptyset$, which is absurd.

 \eqref{Cor:D4}  Any isolated point of $X_c$ would produce an atom of
 $\ropen(X_c)$; therefore  by part~\eqref{Cor:D1}, $X_c$ is a perfect set.
 \end{proof}

\section{Essential Covers and Isomorphisms of the Lattice of Regular
  Open Sets}  \label{sec:ECI}

Suppose $X$ and $Y$ are compact Hausdorff spaces and $(Y, \pi)$ is
an essential cover of $X$.  Letting $(P,f)$ be a projective cover for
$X$, projectivity of $P$ implies the existence of a continuous
map $g:P\rightarrow Y$ with $f=\pi\circ
g$.   (Actually $g$ is unique, see~\cite[Corollary~3.22]{PittsZarikianUnPsExC*In}.)  Since $(P,f)$ and
$(Y,\pi)$ are essential covers, $g$ is necessarily surjective, and 
$(P,g)$ is an essential cover for $Y$.   Therefore, 
$(P,g)$ is a projective cover for $Y$. We thus have a 
commuting diagram as in Figure~\ref{CD}.

Let us use \begin{center}\begin{tabular}{lcl}
$\Psi_f: \ropen(P)\rightarrow
\ropen(X)$, && $\Psi_g:\ropen(P)\rightarrow \ropen(Y),$\\
$\Phi_f:\ropen(X)\rightarrow \ropen(P)$, &and&  
$\Phi_g:\ropen(Y)\rightarrow\ropen(P)$
           \end{tabular}
         \end{center}
         for
the isomorphisms obtained by applying  Proposition~\ref{fineform} to
$(P, f)$ and $(P, g)$ respectively.
Clearly 
\begin{equation}\label{pp1}
  \Psi:= \Psi_f\circ\Phi_g:\ropen(Y)\rightarrow\ropen(X)
  \dstext{and}\Phi:=\Psi_g\circ\Phi_f: \ropen(X)\rightarrow\ropen(Y) 
\end{equation}
are isomorphisms and $\Psi^{-1}=\Phi$. 

The purpose of this short section is to give formulae for $\Psi$ and
$\Phi$ in terms of $\pi$.
\begin{proposition} \label{YXiso} Suppose $X$ and $Y$ are compact Hausdorff spaces
  and $(Y,\pi)$ is an essential cover of $X$.   Then the Boolean algebra isomorphism $\Psi:
  \ropen(Y)\rightarrow \ropen(X)$ of~\eqref{pp1} is given by
  \begin{align}\Psi(U)&=\intr(\pi(\cl (U))), &U\in
    \ropen(Y). \label{YXiso1}\\
  \intertext{
   Furthermore, $\Phi=\Psi^{-1}$ is given by}
  \Phi(V)&=\intr(\cl(\pi^{-1}(V))), & 
    V\in\ropen(X). \label{YXiso2}
  \end{align}
\end{proposition}
\begin{proof}
Throughout the proof, we use the notation discussed in the first paragraphs of this
section.

Let $U\in\ropen(Y)$ and set \[Q:=\Phi_g(U)\stackrel{\eqref{fineform1}}{=}\cl(g^{-1}(U))\in\ropen(P).\]   Let us
show that
\begin{equation}\label{YXiso3}
  f(Q)=\pi(\cl(U)).
\end{equation}
Suppose $x\in \pi(\cl(U))$ and choose  $y\in \cl(U)$ with $\pi(y)=x$.
Then there exists a net $(y_\lambda)$ in $U$ such that
$y_\lambda\rightarrow y$.  Since $f(g^{-1}(U))\subseteq
f(\cl(g^{-1}(U)))$ and $f=\pi\circ g$, we have
\[\pi(U)\subseteq f(\cl (g^{-1}(U))).\]  Therefore, for every
$\lambda$, $\pi(y_\lambda)\in f(\cl(g^{-1}(U)))$.  As
$y_\lambda\rightarrow y$ and $\pi$ is continuous, $x\in
f(\cl(g^{-1}(U)))$.  Thus
\[\pi(\cl(U))\subseteq f(\cl(g^{-1}(U)))=f(Q).\]  To obtain the
reverse inclusion, suppose $x\in f(Q)$.   By definition of $Q$,
there exists $p\in \cl(g^{-1}(U))$ such that $f(p)=x$.  Choose a net
$(p_\lambda)$ in $g^{-1}(U)$ so that $p_\lambda \rightarrow p$.  Then
\[x=f(p)=\lim f(p_\lambda)=\lim \pi(g(p_\lambda))\in \pi(\cl U).\]
We therefore obtain~\eqref{YXiso3}.

We conclude that  $\intr (f(Q))= \intr(\pi(\cl(U)))$, that is, for
every $U\in\ropen(Y)$,
\begin{equation}\label{XYiso3.5}
  \Psi_f(\Phi_g(U))=\Psi(U).
\end{equation}
This gives~\eqref{YXiso1}. 

Turning to~\eqref{YXiso2}, suppose $V\in\ropen(X)$ and
put \[R:=\Phi_f(V) \stackrel{\eqref{fineform1}}{=}\cl(f^{-1}(V))\in\ropen(P).\] We claim that
\begin{equation}\label{XYiso4}
  g(R)=\cl(\pi^{-1}(V)).
\end{equation}
Let $y\in g(R)$ and let $p\in R$ satisfy $g(p)=y$.  We may find a net
$(p_\lambda)$ in $f^{-1}(V)$ with $p_\lambda\rightarrow p$.  Note that
$p_\lambda\in g^{-1}(\pi^{-1}(V))$, because 
$\pi\circ g=f$.  Thus $g(p_\lambda)\in
\pi^{-1}(V)$. Since $g(p_\lambda)\rightarrow g(p)= y$, we obtain
$y\in\cl(\pi^{-1}(V))$.  Therefore,
\[g(R)\subseteq \cl(\pi^{-1}(V)).\]

For the reverse inclusion, suppose
$y\in \cl(\pi^{-1}(V))$.  Choose $y_\lambda\in \pi^{-1}(V)$ so that
$y_\lambda\rightarrow y$.  Since $g$ is surjective, for each
$\lambda$, we may choose $p_\lambda\in g^{-1}(\pi^{-1}(V))$ so that
$g(p_\lambda)=y_\lambda$.  By passing to a subnet if necessary,
compactness of $P$ allows us to 
assume that the net $(p_\lambda)$ converges to $p\in P$.
As  $p_\lambda\in g^{-1}(\pi^{-1}(V))=f^{-1}(V)$, we find $p\in \cl(f^{-1}(V))=R$.  But
$g(p)=\lim g(p_\lambda)=\lim y_\lambda =y$, so $y\in g(R)$.
Therefore~\eqref{XYiso4} holds.

Taking the interiors of both sides of~\eqref{XYiso4} we obtain
\[\Psi_g(\Phi_f(V))=\Phi(V),\] giving~\eqref{YXiso2}.
\end{proof}

\section{Essential Extensions and Isomorphisms of Regular Ideals}\label{sec:BARI}

Proposition~\ref{YXiso} describes a Boolean algebra isomorphism (and
its inverse) of
$\ropen(X)$ onto $\ropen(Y)$ arising
from an irreducible map $\pi: Y\rightarrow X$.  In this section, we
recast this result in terms of unital, abelian \cstaralg s.  In this context,  regular
open sets are replaced by regular ideals, continuous surjections are
replaced by $*$-monomorphisms, and essential covers are
replaced by essential extensions.   The collection of regular ideals
of a unital abelian \cstaralg\ is naturally a Boolean algebra.   When 
 $\A$ and $\B$ are unital, abelian \cstaralg s and
$\alpha: \A\rightarrow \B$ is a $*$-monomorphism which is essential in
the sense that for any
non-zero ideal $K\idealin \B$, $\alpha^{-1}(K)$ is non-zero, we shall
describe a Boolean algebra isomorphism between the regular ideals of
$\A$ and the regular ideals of $\B$.   This is accomplished in
Proposition~\ref{ridealISO}.   While  Proposition~\ref{YXiso} and
Proposition~\ref{ridealISO} are  the same result but in
different categories, we wish to recast the 
Boolean algebra isomorphisms of  Proposition~\ref{YXiso} in terms of
algebraic data.

\subsection*{The Boolean
  Algebra of Regular Ideals}

Once again, $X$ is a compact Hausdorff space.

For any $f\in C(X)$, we denote the open support of $f$ by
 \[\supp(f):=\{x\in X: f(x)\neq 0\}.\]  For an ideal
$J\idealin C(X)$, \[\supp(J):=\{x\in X: f(x)\neq 0 \text{ for some }f\in
J\}=\bigcup\{\supp(f): f\in J\}.\]  Then $\supp(J)$ is an open set in $X$. 

For any open set $G\subseteq X$, let \[\ideal(G):=\{f\in C(X):
f|_{X\setminus G}=0\} \simeq C_0(G).\] 

For any set $S\subseteq C(X)$, the \textit{annihilator} of $S$ is the
set,
\[S^\perp :=\{f\in C(X): fg=0 \text{ for every } g\in S\}.\]   Notice
that $S^\perp$ is an ideal of $C(X)$.  We will use $S^\dperp$ to
denote $(S^\perp)^\perp$.   The ideal $J\idealin C(X)$ is
called a \textit{regular ideal} if $J=J^\dperp$.

\begin{lemma} \label{psup} Let $G\subseteq X$ be open, and let 
$J:=\ideal(G)$.  Then
\begin{align} J^\perp&=\ideal( (X\setminus
  G)^\circ)=\ideal(G^\perp)=\ideal(\neg G) \dstext{and}\label{psup1}\\ 
J^\perp{}^\perp&=\ideal((\overline{G})^\circ)  = \ideal(G^\dperp).\label{psup2}
\end{align}
\end{lemma}
\begin{proof}
  Suppose $h\in C(X)$ and $hf=0$ for every $f\in J$.  Given $x\in G$,
  choose $f\in J$ with $f(x)\neq 0$.  Then $h(x)=0$.  This shows that
  $h|_G =0$, so that $\supp(h)\subseteq (X\setminus G)^\circ.$
  Conversely, if $h\in C(X)$ satisfies
  $\supp(h)\subseteq (X\setminus G)^\circ$, then $hJ=0$, so
  $h\in J^\perp$.  Thus, \eqref{psup1} holds.

The equalities in~\eqref{psup2}  follow from~\eqref{psup1} and the fact that 
$\overline{G}= X\setminus (X\setminus G)^\circ$.
\end{proof}

\begin{corollary} \label{supportSet}
The ideal $J\idealin C(X)$ is regular if and only if
  $\supp(J)$ is a regular open set.
\end{corollary}
\begin{proof} $J^\perp{}^\perp =J$ if and only if $(\overline G)^\circ=
   G^\dperp =G$.
\end{proof}

Now let $\rideal(C(X))$ denote the collection of all regular ideals in
$C(X)$.  For $J,  J_1$ and $J_2$ in $\rideal(C(X))$, define
\begin{equation}\label{ridealops}\neg J:=J^\perp,\quad  J_1\vee J_2:= (J_1\cup J_2)^\dperp, \dstext{and}
  J_1\wedge J_2:=J_1\cap J_2.
\end{equation}
Let 
\begin{equation}\label{suppideal}
\supp: \rideal(C(X))\rightarrow \ropen(X) \dstext{and} \ideal:
\ropen(X)\rightarrow \rideal(C(X))
\end{equation}
denote the maps 
$\rideal(C(X))\ni J\mapsto \supp(J)$ and $\ropen(X)\ni G\mapsto
\ideal(G)$ respectively.

The following fact is surely known, but we do not know a reference.
\begin{lemma}[Folklore]\label{ridisorop}
The map $\supp$ is a bijection of $\rideal(C(X))$
onto $\ropen(X)$ which preserves the operations found in
\eqref{ridealops} and \eqref{ropenops};  further, $(\supp)^{-1}=\ideal$.   In particular,
$\rideal(C(X))$ is a Boolean algebra isomorphic under $\supp$ to
$\ropen(X)$.
\end{lemma}
\begin{proof}  Corollary~\ref{supportSet} shows $\supp$ is bijective
  and $(\supp)^{-1}=\ideal$.
  That meets are preserved is routine.  To check joins, for $j=1,2$, let
  $J_i$ be regular ideals and let $G_i=\supp(J_i)$.  Note that the
  closed linear span of $J_1\cup J_2$ is the ideal whose support is
  $G_1\cup G_2$, so the support of $J_1\vee J_2$ is
  $(\overline{G_1\cup G_2})^\circ$.
  Finally, for $J\in \rideal(C(X))$, Lemma~\ref{psup} gives
  $\supp(J^\perp)=\supp(J)^\perp$.
\end{proof}

\subsection*{Essential Extensions and Isomorphisms of the Lattice of
  Regular Ideals}

Let $\D$ be a unital and abelian \cstaralg.  An \textit{extension of
  $\D$} is defined to be a pair $(\D_1,\alpha)$ consisting of an
abelian unital \cstaralg\ $\D_1$ and a $*$-monomorphism
$\alpha:\D\rightarrow \D_1$.  The extension $(\D_1,\alpha)$ is called
\textit{essential} if the following property holds: whenever
$\theta:\D_1\rightarrow\bh$ is a $*$-representation of $\D_1$ such
that $\theta\circ\alpha$ is a faithful representation of $\D$, then
$\theta$ is faithful.

\begin{remark}{Remark} Notice that $(\D_1,\alpha)$ is an essential
  extension of $\D$ if and only if it has the following property:
  whenever $J\idealin \D_1$ is a non-zero ideal, $\alpha^{-1}(J)$ is a
  non-zero ideal of $\D$.
\end{remark}

Covers for $X$ and extensions for $C(X)$ are in bijective
correspondence, a fact which we now discuss a little more.
 If $(Y,\pi)$ is a cover for $X$ and
$(C(Y),\alpha)$ is an extension for $C(X)$, then $(Y, \pi)$
and  $(C(Y),\alpha)$ are called \textit{dual} if 
\[f\circ\pi=\alpha(f)\dstext{for every} f\in C(X).\]  When this
occurs we will say $(Y,\pi)$ (resp.\ $(C(Y), \alpha)$) is
\textit{dual} to $(C(Y),\alpha)$ (resp. $(Y,\pi)$) and will write
\[\dual{(Y,\pi)}=(C(Y),\alpha)\dstext{and}\dual{(C(Y),\alpha)}=(Y,\pi).\]

For every cover $(Y,\pi)$ of $X$ there is a unique dual extension $(C(Y),\alpha)$:
 take
$\alpha: C(X)\rightarrow C(Y)$ to be the map $C(X)\ni f\mapsto f\circ\pi$.
Likewise, if
$(C(Y),\alpha)$ is a extension of $C(X)$, then there is a unique dual
cover $(Y,\pi)$ for $X$.   Indeed, for every $y\in Y$,
$f\mapsto \alpha(f)(y)$ is a multiplicative linear functional on
$C(X)$; this gives  a unique element $\pi(y)\in X$.  Then  $(Y,\pi)$
is the unique 
cover for $X$ which satisfies $f\circ \pi=\alpha(f)$ for every $f\in C(X)$.

As expected, the dual of an essential cover (resp.\
extension) is an essential extension (resp.\ cover). 
\begin{lemma}\label{ecoveex}  Suppose $(Y,\pi)$ is  cover for $X$.
  Then $(Y,\pi)$ is an essential cover if and only if
  $(C(Y), \alpha):=\dual{(Y,\pi)}$ is an essential extension of
  $C(X)$.
\end{lemma}
\begin{proof}
Suppose $(C(Y),\alpha)$ is an essential extension for $C(X)$ and let 
$Y_0$ be a closed subset of $Y$ such that $\pi(Y_0)=X$.  If $f\in
C(X)$ and $\alpha(f)$ vanishes on $Y_0$, then $f$ vanishes on
$\pi(Y_0)$, that is, $f=0$.   This means that
$\alpha^{-1}(\ideal(Y\setminus Y_0))=\{0\}$.  Since $(C(Y),\alpha)$ is
an essential extension,
$\ideal(Y\setminus Y_0)=\{0\}$, that is, $Y\setminus Y_0=\emptyset$.
Thus $(Y, \pi)$ is an essential cover for $X$.

Conversely, suppose $(Y,\pi)$ is an essential cover for $X$ and let
$J\idealin C(Y)$ be such that $\alpha^{-1}(J)=\{0\}$.    Put
$Y_0:=Y\setminus\supp(J)$.  If $f\in C(X)$ and $\alpha(f)$ vanishes on
$Y_0$, then $\alpha(f)\in J$, whence $f=0$.  It follows that $(Y_0,\pi|_{Y_0})$
is a cover for $X$, for if not, we may choose a non-zero $f\in C(X)$
vanishing on $\pi(Y_0)$.   By assumption, this gives $Y_0=Y$, that is, $J=\{0\}$.  Thus
$(C(Y),\alpha)$ is an essential extension of $C(X)$.
\end{proof}

Now suppose $(Y,\pi)$ is an essential cover for $X$, with dual
extension $(C(Y),\alpha)$. 
For $E\in \{X,Y\}$, let $\supp_E$ and $\ideal_E$ be
  the isomorphisms of $\rideal(C(E))$ onto $\ropen(E)$ and $\ropen(E)$
  onto $\rideal(C(E))$ described in Lemma~\ref{ridisorop}.  Also, let
  $\Phi$ and $\Psi$ be the isomorphisms described in
  Proposition~\ref{YXiso}.
  Then the Boolean algebra isomorphisms
  \begin{align}
   \Upsilon:= (\ideal_Y)\circ\Phi\circ(\supp_X):
    \rideal(C(X))\rightarrow\rideal(C(X)) \label{si1}\\\intertext{and}
   \Omega:= (\ideal_X)\circ\Psi\circ(\supp_Y):\rideal(C(Y))\rightarrow\rideal(C(X))\label{si2}
  \end{align} satisfy $\Upsilon^{-1}=\Omega$.  The remainder of this
  section is devoted to finding formulae for $\Upsilon$ and $\Omega$.
  This is accomplished in Proposition~\ref{ridealISO}.

\begin{lemma}\label{ropen->rideal}  Suppose $(Y,\pi)$ is an essential
  cover for $X$, let $(C(Y),\alpha):=\dual{(Y,\pi)}$, and let
  $\Psi:\ropen(Y)\rightarrow\ropen(X)$ and
  $\Phi:\ropen(X)\rightarrow\ropen(Y)$ be the Boolean algebra
  isomorphisms described in Proposition~\ref{YXiso}.
    The following statements hold.
\begin{enumerate}
  \item\label{ro->ri1}  
  For $J\in \rideal(C(X))$, $\supp(\alpha(J)^\dperp) = \Phi(\supp(J))$.
\item \label{ro->ri2}  For $K\in \rideal(C(Y))$,
  $\supp(\alpha^{-1}(K))=\Psi(\supp(K))$.
\end{enumerate}
\end{lemma}
\begin{proof}
 \eqref{ro->ri1} Let $J\in \rideal(C(X))$ and let $V=\supp(J)\in
 \ropen(X)$.  Since
  $\alpha(J)=\{f\circ\pi: f\in J\}$,
  \begin{equation}\label{ro->ri2.5}
    \alpha(J)^\perp=\{g\in C(Y):g(y) f(\pi(y))=0\text{ for all $f\in
    J$ and $y\in Y$}\}.
\end{equation}
Suppose $g\in \alpha(J)^\perp$.  Let us show  $g$ vanishes on
$\pi^{-1}(V)$.  Given $y\in
\pi^{-1}(V)$, we may choose $f\in J$
  so that $f(\pi(y))=1$.  Then $g(y)=g(y)f(\pi(y))=0$.  We conclude
  that  every
  function in  $\alpha(J)^\perp$ vanishes on $\pi^{-1}(V)$, as
  claimed.
  
On the
  other hand, if $g\in C(Y)$ vanishes on $\pi^{-1}(V)$, then $g\in
  \alpha(J)^\perp$.  Thus~\eqref{ro->ri2.5} gives, 
  \[\alpha(J)^\perp=\{g\in C(Y): g|_{\pi^{-1}(V)}=0\}.\]   It follows
  that
  \begin{equation}\label{ro->ri3}
    \supp(\alpha(J)^\perp) = Y\setminus \cl(\pi^{-1}(V)).
  \end{equation}
  Since $\alpha(J)^\perp$ is a regular ideal, its support set belongs to
$\ropen(Y)$ by Corollary~\ref{supportSet}.  Hence $\cl(\pi^{-1}(V))=Y\setminus\supp(\alpha(J)^\perp)$ is a regular closed set.
Thus 
\begin{equation}\label{ro->ri4} Y\setminus \cl(\pi^{-1}(V))=Y\setminus
  \overline{\Phi(V)}=\neg\Phi(V).
\end{equation}
Combining~\eqref{ro->ri3} and~\eqref{ro->ri4} with Lemma~\ref{ridisorop} yields,
  \[\alpha(J)^\perp= \ideal(\neg\Phi(V)) =\neg\ideal(\Phi(V)).\]
  Therefore,
  $\alpha(J)^\dperp= \neg\neg
  \ideal(\Phi(V))=\ideal(\Phi(V))=\ideal(\Phi(\supp(J)))$.  Another
  application of Lemma~\ref{ridisorop} gives \eqref{ro->ri1}.

  \eqref{ro->ri2} Let $K\in \rideal(C(Y))$, set \[U=\supp(K), \dstext{and}
  S=\supp(\alpha^{-1}(K)).\]  For $x\in S$, find $f\in \alpha^{-1}(K)$
  such that $f(x)=1$.  For any $y\in Y$ with $\pi(y)=x$ we have
  $\alpha(f)(y)=f(\pi(y))=1$.  Recalling $\alpha(f)\in K$, we see $y\in U$.  Hence
  $x=\pi(y)\in \pi(U)\subseteq \pi(\cl (U))$.  As $S$ is open, we find
  $S\subseteq \intr(\pi(\cl(U)))=\Psi(U)$, that is,
\[\supp(\alpha^{-1}(K))\subseteq \Psi(\supp(K)).\]

For the reverse inclusion, we first show 
\begin{equation}\label{ro->ri5} \ideal(\Psi(U))\subseteq
  \alpha^{-1}(K).
\end{equation}
Suppose  
$f\in \ideal(\Psi(U))$.
If $y\in Y$ and
$f(\pi(y))\neq 0$, then 
$\pi(y)\in \Psi(U)$. 
Therefore,
\[\supp(f\circ\pi)\subseteq \pi^{-1}(\Psi(U))\subseteq
  \intr(\cl(\pi^{-1}(\Psi(U))))=\Phi(\Psi(U))=U.\] 
This shows that
$\alpha(f)\in K$, so $f\in \alpha^{-1}(K)$, establishing~\eqref{ro->ri5}.
Upon applying the map $\supp$ to each side of~\eqref{ro->ri5},
Lemma~\ref{ridisorop} shows that
\[\Psi(\supp(K))\subseteq \supp(\alpha^{-1}(K)).\]  This completes
the proof.
\end{proof} 
\begin{proposition}  \label{ridealISO}
Suppose 
$(C(Y),\alpha)$ is an essential extension of $C(X)$.   
Then the Boolean algebra isomorphisms  $\Upsilon: \rideal(C(X))\rightarrow\rideal(C(Y))$ of~\eqref{si1} and
  $\Omega:\rideal(C(Y))\rightarrow \rideal(C(X))$ of~\eqref{si2} are
  given by \begin{align}
    \Upsilon(J)&=\alpha(J)^\dperp \qquad J\in\rideal(C(X))
\label{ridealISO1}\\
             \intertext{and}
\Omega(K)&=\alpha^{-1}(K)\qquad K\in\rideal(C(Y)).\label{ridealISO2}
           \end{align}
           \end{proposition}
\begin{proof} Let $(Y,\pi)$ be the essential cover of $X$ dual to
  $(C(Y),\alpha)$.  
  By Lemma~\ref{ropen->rideal}, for $J\in \rideal(C(X))$ and $K\in\rideal(C(Y))$
  \[\alpha(J)^\dperp=\ideal_Y(\Phi(\supp_X(J)) \dstext{and}
    \alpha^{-1}(K)=\ideal_X(\Psi(\supp_Y(K))).\] 
\end{proof}

\section{Boolean Equivalence}\label{SecBE}

\begin{definition}  Call two compact Hausdorff spaces $X$ and $Y$
  \textit{Boolean equivalent} if there is a Boolean algebra isomorphism
  between the Boolean algebras, $\ropen(X)$ and $\ropen(Y)$. 
Similarly, we say that the unital, abelian \cstaralg s $\A$ and $\B$
are \textit{Boolean equivalent} if $\rideal(\A)$ and $\rideal(B)$ are
isomorphic Boolean algebras. 

 We will sometimes  use the notation $X\equiv_B Y$ to indicate that
  $X$ is Boolean equivalent to $Y$.
  \end{definition}

\begin{proposition}\label{BE} Let $X$ and $Y$ be compact Hausdorff spaces with
  projective covers $(P_X,\pi_X)$ and $(P_Y,\pi_Y)$ respectively.
  The following statements are equivalent.
  \begin{enumerate}
    \item The spaces $X$ and $Y$ are
      Boolean equivalent.
    \item The spaces $P_X$ and $P_Y$ are homeomorphic.
    \item  There exists a compact Hausdorff space $Z$ and irreducible
      maps $f: Z\rightarrow X$ and $g: Z\rightarrow Y$.
    \end{enumerate}
  \end{proposition}
\begin{proof}
(a)$\Rightarrow$(b)\  If $X$ and $Y$ are Boolean equivalent,
Proposition~\ref{fineform} implies that $\ropen(P_X)$ and $\ropen(P_Y)$
are isomorphic Boolean algebras.  So $P_X$ is homeomorphic to $P_Y$ by
Stone's theorem.

(b)$\Rightarrow$(c)\  Let $h: P_X\rightarrow P_Y$ be a homeomorphism.
Take $Z=P_X$, $f=\pi_X$ and $g:=\pi_Y\circ h$.

(c)$\Rightarrow$(a)\ 
Apply
Proposition~\ref{YXiso}  to the essential covers $(Z,f)$ and
$(Z,g)$ of $X$ and $Y$ respectively to obtain isomorphisms
$\Psi_{X,Z}: \ropen(Z)\rightarrow\ropen(X)$ and
$\Phi_{Z,Y}:\ropen(Y)\rightarrow \ropen(Z)$.  Then
$\Psi_{X,Z}\circ\Phi_{Z,Y}$ is an isomorphism of $\ropen(Y)$ onto $\ropen(X)$.
\end{proof}

Boolean equivalence is a very weak notion.
We now present some results showing some examples of spaces and
algebras which are
Boolean equivalent.

\begin{flexstate}{Proposition}{}\label{cmtify}  Let $X$ be a locally
  compact, but not compact,
  Hausdorff space  and denote by $X^+$ its one-point compactification.
  Suppose $Y$ is a compactification of $X$, that is, $Y$ is compact
  and $h: X\rightarrow Y$ is an
  embedding with $h(X)$ dense in $Y$.   Then $X^+$ and $Y$ are Boolean
  equivalent.
\end{flexstate}
\begin{proof}
Regard $X$ as a subset of $X^+=X\cup\{\infty\}$.  Since $X$ and $h(X)$
(where $h(X)$ is equipped with the subspace topology)
are homeomorphic, the map
$\pi:Y\rightarrow X^+$ given by \begin{equation}\label{cmtify1} \pi(y)=\begin{cases} x&
  \text{ if } y=h(x)\\
  \infty& \text{ if } y\notin h(X)
\end{cases}
\end{equation}
is a continuous surjection.   
Thus the result will follow from Proposition~\ref{YXiso} once we verify
that $(Y,\pi)$ is an essential cover for $X^+$.    To see this,
suppose $F\subseteq Y$ is closed and $\pi(F)=X$.  Suppose $x\in X$.
By construction, $\pi^{-1}(\{x\})=\{h(x)\}$.  Therefore $h(X)\subseteq
F$, so since $F$ is closed, $F=Y$.  Thus $(Y,\pi)$ is an essential
cover for $X^+$.  
\end{proof}

\begin{corollary}\label{malg}   Suppose $\A$ is an abelian, but
  non-unital, \cstaralg, and denote by  $\tilde\A$ and $M(\A)$ its
  unitization and multiplier algebra respectively. 
  Then  $\tilde\A$ and $M(\A)$ are Boolean equivalent.
\end{corollary}
\begin{proof}
Using the Gelfand theorem, we may assume $A=C_0(X)$ for a suitable  locally compact, but not
compact, Hausdorff space $X$.  Then $C(X^+)=\tilde\A$ and $C(\beta X)=M(\A)$.  Now
  apply Propositions~\ref{cmtify} and Lemma~\ref{ridisorop}.
\end{proof}

Here is a 
result which implies that any two perfect and compact metric spaces are
Boolean equivalent.

\begin{flexstate}{Proposition}{}\label{CantorCmptPMS}
Let $C$ be the Cantor set and suppose $X$ is a  compact and perfect metric space.  Then $X$
and $C$ are Boolean equivalent.
\end{flexstate}
\begin{proof}
The Hausdorff-Alexandroff theorem yields a continuous surjection $\pi:
C\rightarrow X$.   Let $Y\subseteq C$ be a minimal closed set such
that $\pi(Y)=X$.   Then $Y$ is compact and totally disconnected.  Let us
show $Y$ is perfect.

Suppose $y\in Y$ is isolated and put $Y_0:=Y\setminus \{y\}$.   Then
$Y_0$ is closed, hence compact.   Put $x=\pi(y)$.  Since $X$ is perfect,
we may find a sequence $(x_n)$ of distinct elements of $X$ such that
$0<d(x_n,x)\rightarrow 0$.  Let $y_n\in Y$ satisfy
$\pi(y_n)=x_n$.   Then $y_n\in Y_0$, so by compactness of $Y_0$, we may
find a convergent subsequence $(y_{n_k})$, say $y_{n_k}\rightarrow
y_0$.  Clearly $y_0\in Y_0$ and continuity of $\pi$ gives
$\pi(y_0)=\lim \pi(y_{n_k})=\lim x_{n_k}=x$.  This shows $\pi(Y_0)=X$,
contradicting minimality of $Y$.   So $Y$ has no isolated points,
whence $Y$ is perfect.

Brouwer's theorem shows that any non-empty, perfect, totally disconnected metric
space is homeomorphic to the Cantor set.  Thus $(Y, \pi|_{Y})$ is an
essential cover of $X$ by a set homeomorphic to the Cantor set.  Proposition~\ref{YXiso} now
shows the Cantor set and $X$ are Boolean equivalent.
\end{proof}

The following result falls into the same class of results as Kuratowski's classification
of standard Borel spaces~\cite[Theorem~3.3.13]{SrivastavaACBoSe} and Maharam's
decomposition of complete measure spaces into atomic and non-atomic
parts~\cite{MaharamOnHoMeAl}.

\begin{theorem}\label{compactmetric}   Suppose $X$ and $Y$ are  compact
  metric spaces.  Then $X$ and $Y$ are Boolean equivalent if and only
  if $\isol(X)$ has the same cardinality as $\isol(Y)$ and either
  $X_c=Y_c=\emptyset$ or both $X_c$ and $Y_c$ are non-empty.
\end{theorem}
\begin{proof} 
Suppose $X$ and $Y$ are Boolean equivalent.   
Proposition~\ref{BE} shows $P_X$ and $P_Y$ are homeomorphic.  Therefore, 
$(P_X)_c$ is homeomorphic to $(P_Y)_c$, and, by
Corollary~\ref{Cor:D}\eqref{Cor:D3},
 the cardinalities of 
 $\isol(X)$ and $\isol(Y)$ are the same.  Then
 \[\ropen(X_c)\stackrel{\eqref{reclBA}}{\simeq} \ropen((P_X)_c)\simeq
   \ropen((P_Y)_c) \stackrel{\eqref{reclBA}}{\simeq} \ropen(X_c).\]
 Thus $X_c\equiv_B X_c$, so $X_c$ and $Y_c$ are either both empty or
 both non-empty.  

For the converse, suppose $\isol(X)$ has the same cardinality as
$\isol(Y)$ and  $X_c$ and $Y_c$ are either both empty or both non-empty.
Proposition~\ref{cmtify} gives $X_a\equiv_B Y_a$.  By
Corollary~\ref{Cor:D}\eqref{Cor:D4}, $X_c$ and $Y_c$ are perfect
metric spaces,
so Proposition~\ref{CantorCmptPMS} gives $X_c\equiv_B Y_c$.
Therefore $\ropen(X_a)\times \ropen(X_c)$ is isomorphic to
$\ropen(Y_a)\times \ropen(Y_c)$.   Corollary~\ref{Cor:D}\eqref{Cor:D2}
then gives $X\equiv_B Y$. 
\end{proof}

\def\cprime{$'$}
\providecommand{\bysame}{\leavevmode\hbox to3em{\hrulefill}\thinspace}
\providecommand{\MR}{\relax\ifhmode\unskip\space\fi MR }
\providecommand{\MRhref}[2]{%
  \href{http://www.ams.org/mathscinet-getitem?mr=#1}{#2}
}
\providecommand{\href}[2]{#2}

\end{document}